\documentclass[12pt,reqno]{amsart}

\usepackage{color}
\usepackage{amsmath, amsthm, amssymb}
\usepackage{amsfonts}

\usepackage[ansinew]{inputenc}
\usepackage[dvips]{epsfig}
\usepackage{graphicx}
\usepackage[english]{babel}
\usepackage{hyperref}
\theoremstyle{plain}
\newtheorem{thm}{Theorem}[section]
\newtheorem{cor}[thm]{Corollary}
\newtheorem{lem}[thm]{Lemma}
\newtheorem{prop}[thm]{Proposition}

\theoremstyle{definition}

\theoremstyle{remark}
\newtheorem{rem}[thm]{Remark}

\numberwithin{equation}{section}

\newcommand{\average}{{\mathchoice {\kern1ex\vcenter{\hrule height.4pt
width 6pt depth0pt} \kern-9.7pt} {\kern1ex\vcenter{\hrule
height.4pt width 4.3pt depth0pt} \kern-7pt} {} {} }}

\def\R{\mathbb{R}}

\begin{document}

\title[{Nonlocal equations in bounded domains: a survey}]{Nonlocal elliptic equations in bounded domains: \\ a survey}

\author{Xavier Ros-Oton}
\address{The University of Texas at Austin, Department of Mathematics, 2515 Speedway, Austin, TX 78751, USA}
\email{ros.oton@math.utexas.edu}

\thanks{The author was supported by grants MTM2011-27739-C04-01 (Spain), and 2009SGR345 (Catalunya)}


\maketitle

\begin{abstract}
In this paper we survey some results on the Dirichlet problem
\[\left\{ \begin{array}{rcll}
L u &=&f&\textrm{in }\Omega \\
u&=&g&\textrm{in }\R^n\backslash\Omega
\end{array}\right.\]
for nonlocal operators of the form
\[Lu(x)=\textrm{PV}\int_{\R^n}\bigl\{u(x)-u(x+y)\bigr\}K(y)dy.\]
We start from the very basics, proving existence of solutions, maximum principles, and constructing some useful barriers.
Then, we focus on the regularity properties of solutions, both in the interior and on the boundary of the domain.

In order to include some natural operators $L$ in the regularity theory, we do not assume any regularity on the kernels.
This leads to some interesting features that are purely nonlocal, in the sense that have no analogue for local equations.

We hope that this survey will be useful for both novel and more experienced researchers in the field.
\end{abstract}

\vspace{8mm}

\tableofcontents

\newpage

\section{Introduction}

The aim of this paper is to survey some results on Dirichlet problems of the form
\begin{equation}\label{pb}
\left\{ \begin{array}{rcll}
L u &=&f&\textrm{in }\Omega \\
u&=&g&\textrm{in }\R^n\backslash\Omega.
\end{array}\right.\end{equation}
Here, $\Omega$ is any bounded domain in $\R^n$, and $L$ is an elliptic integro-differential operator of the form
\begin{equation}\label{L}
Lu(x)=\textrm{PV}\int_{\R^n}\bigl\{u(x)-u(x+y)\bigr\}K(y)dy.
\end{equation}
The function $K(y)\geq0$ is the kernel of the operator\footnote{The typical example is $K(y)=c_{n,s}|y|^{-n-2s}$, which corresponds to $L=(-\Delta)^s$, the fractional Laplacian.
As $s\uparrow1$, it converges to the Laplacian $-\Delta$.}
\footnote{For equations with $x$-dependence, the kernel is a function of two variables, $K(x,y)$. Here, we assume that the operator $L$ is translation invariant, and thus that the kernel $K$ does not depend on $x$.}, and we assume
\[K(y)=K(-y)\qquad\textrm{and}\qquad \int_{\R^n}\min\bigl\{|y|^2,1\bigr\}K(y)dy<\infty.\]

Integro-differential problems of the form \eqref{pb} arise naturally in the study of stochastic processes with jumps, and have been widely studied both in Probability and in Analysis and PDEs.
We refer to the reader to the following works on:
\begin{itemize}
\item Existence of solutions: Felsinger-Kassmann-Voigt \cite{FKV}, Hoh-Jacob \cite{HJ}, and Barles-Imbert \cite{BI};
\item Interior regularity of solutions: Bass-Levin \cite{Bass2}, Kassmann \cite{K}, Caffarelli-Silvestre \cite{CS,CS2,CS3}, Barles-Chasseigne-Imbert \cite{BCI}, Kassmann-Mimica \cite{KM,KM2}, Schwab-Silvestre \cite{SS}, and Serra \cite{Se2};
\item Boundary regularity of solutions: Bogdan \cite{Bogdan}, the author and Serra \cite{RS-Dir,RS-K,RS-stable}, Grubb \cite{Grubb,Grubb2}, Chen-Song \cite{Potential2}, Barles-Chasseigne-Imbert \cite{BCI2}, Bogdan-Grzywny-Ryznar \cite{BGR2}, and Bogdan-Kumagai-Kwa\'snicki \cite{BKK};
\item Other qualitative properties of solutions: Birkner-L\'opez-Wakolbinger \cite{BMW}, Dipierro-Savin-Valdinoci \cite{DSV}, Abatangelo \cite{Abatangelo}, and Kulczycki \cite{Kulcz}.
\end{itemize}

\vspace{1mm}

In applications, this type of nonlocal problems appear in models involving ``anomalous diffusions'' (in which the underlying stochastic process is not given by Brownian motion), or in the presence of long-range interactions or forces.
In particular, they appear in Physics \cite{DGLZ,ZD,phys3,phys2,Eringen,Laskin}, Finance \cite{Schoutens,finance3,AB}, Fluid dynamics \cite{DG-V,fluids}, Ecology \cite{MV,nature,nature2}, or Image processing \cite{image3}.

\vspace{3mm}

The aim of this paper is to survey some results on existence, regularity, and qualitative properties of solutions to \eqref{pb}.
More precisely, we will start from the very basics (showing existence and boundedness of solutions), to then focus on the regularity of solutions, both in the interior and on the boundary of the domain $\Omega$.

In order to include some natural operators $L$ in the regularity theory, we do \emph{not} assume any regularity on the kernel $K(y)$.
As we will see, there is an interesting relation between the regularity properties of solutions and the regularity of the kernels $K(y)$.
This is a purely nonlocal issue, in the sense that has no analogue in second order equations, as explained next.

For linear elliptic second order equations of the form
\begin{equation}\label{pb-local}
\left\{ \begin{array}{rcll}
-\sum_{ij} a_{ij}\partial_{ij}u &=&f&\textrm{in }\Omega \\
u&=&g&\textrm{on }\partial\Omega,
\end{array}\right.\end{equation}
the interior regularity properties of $u$ depend only on the regularity of $f$.
This is because, after an affine change of variables, this equation is just $-\Delta u=f$.

The nonlocal analogue of \eqref{pb-local} is \eqref{pb}.
For this problem, the interior regularity of solutions depends on the regularity of $f$ ---as in \eqref{pb-local}---, but it also depends on the regularity of $K(y)$ in the $y$-variable.
Furthermore, if the kernel $K$ is not regular, then the interior regularity of $u$ will in addition depend on the regularity of $g$, on the boundary regularity of $u$, and even on the shape of $\Omega$, as we will see later on.

Hence, for nonlocal equations, the class of linear and translation invariant operators is much richer, and already presents several interesting features.
The same type of issues appear when considering nonlinear equations, or operators with $x$-dependence.
However, for the clarity of presentation, we will consider only linear and translation invariant equations.

In most of the regularity results we will focus on two classes of kernels: we will assume $s\in(0,1)$ and either that
\begin{equation}\label{rough}
\frac{\lambda}{|y|^{n+2s}}\leq K(y)\leq \frac{\Lambda}{|y|^{n+2s}},\qquad 0<\lambda \leq \Lambda;
\end{equation}
or that
\begin{equation}\label{stable}
\hspace{25mm} K(y)=\frac{a\left(y/|y|\right)}{|y|^{n+2s}},\qquad\qquad a\in L^1(S^{n-1}),\quad a\geq0.
\end{equation}
In fact, in order to include operators like $L=(-\partial_{x_1x_1}^2)^s+\cdots+(-\partial_{x_nx_n}^2)^s$, we will allow $a$ in \eqref{stable} to be any nonnegative measure on $S^{n-1}$.
The only necessary condition on the measure $a$ is that it is not supported in a hyperplane ---so that the operator $L$ is not degenerate.
A quantitative way to state this nondegeneracy condition is that, for some positive constants $\lambda$ and $\Lambda$, one has $\int_{S^{n-1}}da\leq \Lambda$ and $\int_{S^{n-1}}|\nu\cdot\theta|^{2s}da(\theta)\geq \lambda$ for all $\nu\in S^{n-1}$; see \cite{RS-stable}.

\vspace{2mm}

The paper is organized as follows.
In Section~\ref{kernels} we briefly explain the probabilistic interpretation of \eqref{pb}.
In Section~\ref{existence} we show existence of weak solutions.
In Section~\ref{max-princ} we prove the maximum principle for such solutions.
In Section~\ref{barriers} we construct some useful barriers and give an $L^\infty$ estimate.
In Section~\ref{interior} we discuss the interior regularity properties for the classes of kernels \eqref{rough} and \eqref{stable}.
In Section~\ref{boundary} we see what is the boundary regularity of solutions.
Finally, in Section~\ref{further-interior} we come back to the interior regularity of solutions to \eqref{pb}.

\section{Motivation and some preliminaries}
\label{kernels}

\addtocontents{toc}{\protect\setcounter{tocdepth}{1}}  
\subsection{L\'evy processes}

Integro-differential equations of the form \eqref{pb} arise naturally in the study of stochastic processes with jumps, and more precisely in L\'evy processes.
A L\'evy process is a stochastic process with independent and stationary increments.
Informally speaking, it represents the random motion of a particle whose successive displacements are independent and statistically identical over different time intervals of the same length.

These processes extend the concept of Brownian motion, and were introduced a few years after Wiener gave the precise definition of the Brownian motion \cite{Levy}.
Essentially, L\'evy processes are obtained when one relaxes the assumption of continuity of paths (which gives the Brownian motion) by the weaker assumption of stochastic continuity.

By the L\'evy-Khintchine Formula, the infinitesimal generator of any L\'evy processes is an operator of the form\footnote{We denote the infinitesimal generator $-L$ in order to be consistent with our notation in \eqref{L}. We always take $L$ to be positive definite, so it is the analogue of $-\Delta$.}
\[-Lu(x)=\sum_{i,j} a_{ij}\partial_{ij}u+\sum_j b_j\partial_j u+\int_{\R^n}\bigl\{u(x+y)-u(x)-y\cdot\nabla u(x)\chi_{B_1}(y)\bigr\}d\nu(y),\]
where $\nu$ is the L\'evy measure, and satisfies $\int_{\R^n}\min\bigl\{1,|y|^2\bigr\}d\nu(y)<\infty$.
When the process has no diffusion or drift part, this operator takes the form
\[-Lu(x)=\int_{\R^n}\bigl\{u(x+y)-u(x)-y\cdot\nabla u(x)\chi_{B_1}(y)\bigr\}d\nu(y).\]
Furthermore, if one assumes the process to be symmetric, and the L\'evy measure to be absolutely continuous, then $L$ can be written as \eqref{L} or, equivalently,
\[Lu(x)=\frac12\int_{\R^n}\bigl\{2u(x)-u(x+y)-u(x-y)\bigr\}K(y)dy,\]
with $K(y)=K(-y)$.
We will use this last expression for $L$ throughout the paper.

As an example, let $\Omega\subset\R^n$ be any bounded domain, and let us consider a L\'evy process $X_t$, $t\geq0$, starting at $x\in\Omega$.
Let $u(x)$ be the expected first exit time, i.e., the expected time $\mathbb E[\tau]$, where $\tau=\inf\{t>0\,:\,X_t\notin\Omega\}$ is the first time at which the particle exits the domain.
Then, $u(x)$ solves
\[\left\{ \begin{array}{rcll}
L u &=&1&\textrm{in }\Omega \\
u&=&0&\textrm{in }\R^n\backslash\Omega,\end{array}\right.\]
where $-L$ is the infinitesimal generator of $X_t$.

Recall that, when $X_t$ is a Brownian motion, then $L$ is the Laplace operator $-\Delta$.
In the context of integro-differential equations, L\'evy processes plays the same role that Brownian motion play in the theory of second order equations.

Another simple example is given by the following.
Let us also consider a bounded domain $\Omega$, a process $X_t$ starting at $x\in \Omega$, and the first time $\tau$ at which the particle exits the domain.
Assume now that we have a payoff function $g\,:\,\R^n\setminus\Omega\longrightarrow\R$, so that when the process $X_t$ exits $\Omega$ we get a payoff $g(X_\tau)$.
Then, the expected payoff $u(x):=\mathbb E[g(X_\tau)]$ solves the problem
\[\left\{ \begin{array}{rcll}
L u &=&0&\textrm{in }\Omega \\
u&=&g&\textrm{in }\R^n\backslash\Omega.\end{array}\right.\]

The Dirichlet problem \eqref{pb} arise when considering at the same time a running cost $f$ and a final payoff $g$.

\subsection{Kernels with compact support}

It is important to remark that when the kernel $K(y)$ has compact support in a ball $B_\delta$ (for some $\delta>0$), then the Dirichlet problem is
\[\left\{ \begin{array}{rcll}
L u &=&f&\textrm{in }\Omega \\
u&=&g&\textrm{in }(\Omega+B_\delta)\backslash\Omega.\end{array}\right.\]
This means that the exterior condition $g$ has to be posed only in a neighborhood of $\partial\Omega$ and not in the whole $\R^n\backslash\Omega$.
This turns out to be convenient in some applications; see for example \cite{DGLZ,Silling,ZD,CPQ}.

However, from the analytical point of view, it is only the singularity of the kernel $K(y)$ at the origin that plays a role in the regularity properties of solutions; see for example Section 14 in \cite{CS}.

\subsection{Stable processes}

A special class of L\'evy processes are the so-called stable processes.
These are the processes which satisfy self-similarity properties, and they are also the ones appearing in the Generalized Central Limit Theorem; see \cite{ST}.

The infinitesimal generators of these processes are given by \eqref{L}-\eqref{stable}.
Note that the structural condition \eqref{stable} on the kernel $K$ is equivalent to saying that the L\'evy measure is \emph{homogeneous}.
This is also equivalent to the fact that the operator $L$ is scale invariant.

A very natural stable process is the radially symmetric one.
This means that $K(y)=c|y|^{-n-2s}$ and hence, up to a multiplicative constant,
\begin{equation}\label{frac-lap}
L=(-\Delta)^s.
\end{equation}

Another very natural stable process is the one obtained by taking independent stable processes in each coordinate.
That is, we consider $X_t=(X_t^{(1)},...,X_t^{(n)})$, where $X_t^{(i)}$ are 1-dimensional i.i.d. symmetric stable processes.
The generator of this process $X_t$ will be
\begin{equation}\label{frac-der}
L=(-\partial_{x_1x_1}^2)^s+\cdots+(-\partial_{x_nx_n}^2)^s,
\end{equation}
and it corresponds to \eqref{stable} with the measure $a$ being $2n$ delta functions on $S^{n-1}$.
For example, when $n=2$, one has $a=\delta_{(0,1)}+\delta_{(0,-1)}+\delta_{(1,0)}+\delta_{(-1,0)}$.

While the operators \eqref{frac-lap} and \eqref{frac-der} may look similar (they are the same when $s=1$), they have quite different regularity properties, as we will see in the next sections.
For example, while solutions to $(-\Delta)^su=0$ in $\Omega$ are $C^\infty$ inside the domain, this may not be the case for \eqref{frac-der}.

Note also that the Fourier symbols of \eqref{frac-lap} and \eqref{frac-der} are $|\xi|^{2s}$ and $|\xi_1|^{2s}+\cdots+|\xi_n|^{2s}$, respectively.
Again, the first one is $C^\infty$ outside the origin, while the second one is just $C^{2s}$.
In general, the Fourier symbol of any stable process of the form \eqref{stable} is
\begin{equation}\label{symbol-stable}
A(\xi)=c\int_{S^{n-1}}|\xi\cdot\theta|^{2s}a(\theta)d\theta.
\end{equation}
The symbol $A(\xi)$ is in general only $C^{2s}$ outside the origin, but it would be $C^\infty$ outside the origin whenever the function $a$ is $C^\infty$ on $S^{n-1}$.

\section{Existence of solutions}
\label{existence}

In this Section we explain briefly how to prove existence of weak solutions to \eqref{pb} in a simple case.
We hope that this will be useful for students and/or for readers which are not familiar with nonlocal operators.
The more experienced reader should go to Felsinger-Kassmann-Voigt \cite{FKV}, where this is done in a much more general setting.

\vspace{2mm}

As we will see, the existence (and uniqueness) of weak solutions to \eqref{pb}-\eqref{L} follows from the Riesz representation theorem once one has the appropriate ingredients.

The energy functional associated to the problem \eqref{pb} is
\begin{equation}\label{functional}
\mathcal E(u)=\frac14\int\int_{\R^{2n}\setminus(\Omega^c\times\Omega^c)}\bigl(u(x)-u(z)\bigr)^2K(z-x)dx\,dz-\int_\Omega fu.
\end{equation}
The minimizer of $\mathcal E$ among all functions with $u=g$ in $\R^n\setminus\Omega$ will be the unique weak solution of \eqref{pb}.

Notice that $\mathcal E(u)$ is defined for all regular enough functions $u$ which are bounded at infinity.
When $g\equiv0$, then making the change of variables $y=x+z$ the functional can be written as
\begin{equation}\label{functional2}
\mathcal E(u)=\frac14\int_{\R^n}\int_{\R^n}\bigl(u(x)-u(x+y)\bigr)^2K(y)dx\,dy-\int_\Omega fu.
\end{equation}
When $g$ is not zero, the term $\int\int_{\Omega^c\times\Omega^c}|g(x)-g(z)|^2K(z-x)dx\,dz$ could be infinite, and this is why in general one has to take \eqref{functional}.

For simplicity, we will show existence of solutions for the case $g\equiv0$, and hence we can think on the energy functional \eqref{functional2}.

Let $H_K(\R^n)$ be the space of functions $u\in L^2(\R^n)$ satisfying
\begin{equation}\label{seminorm-K}
[u]_{H_K}^2:=\frac12\int_{\R^n}\int_{\R^n}\bigl(u(x)-u(x+y)\bigr)^2K(y)dx\,dy<\infty,
\end{equation}
and let
\[X=\left\{u\in H_K(\R^n)\,:\, u\equiv0\quad \textrm{in}\ \R^n\setminus\Omega\right\}.\]
Essentially, the only assumption which is needed in order to prove existence of solutions is the Poincar\'e inequality
\begin{equation}\label{Poincare}
\int_\Omega u^2\leq C\int_{\R^n}\int_{\R^n}\bigl|u(x)-u(x+y)\bigr|^2K(y)dx\,dy
\end{equation}
for functions $u\equiv0$ in $\R^n\setminus\Omega$.

When the kernel $K$ satisfies \eqref{rough}, then the seminorm in \eqref{seminorm-K} is equivalent to
\begin{equation}\label{seminorm-H^s}
[u]_{H^s}^2:=\frac12\int_{\R^n}\int_{\R^n}\frac{\bigl|u(x)-u(x+y)\bigr|^2}{|y|^{n+2s}}\,dx\,dy.
\end{equation}
In this case, the Poincar\'e inequality \eqref{Poincare} follows easily from the fractional Sobolev inequality\footnote{Different proofs of the fractional Sobolev inequality can be found in \cite{Stein}, \cite{Ponce}, and \cite{DPV}.
The proof given in \cite{Ponce}, which is due to Brezis, is really nice and simple.}
 in $\R^n$ and H\"older's inequality in $\Omega$.

Once one has the Poincar\'e inequality \eqref{Poincare}, then it follows that the space $X$ is a Hilbert space with the scalar product
\[(v,w)_K:=\frac12\int_{\R^n}\int_{\R^n}\bigl(v(x)-v(x+y)\bigr)\bigl(w(x)-w(x+y)\bigr)K(y)dx\,dy.\]
Then, the weak formulation of \eqref{pb} is just
\begin{equation}\label{weak}
(u,\varphi)_K=\int_\Omega f\varphi\qquad \textrm{for all}\ \varphi\in X,
\end{equation}
and the existence and uniqueness of weak solution follows immediately from the Riesz representation theorem.

\vspace{2mm}

For more general classes of kernels $K(y)$, one has to show the Poincar\'e inequality \eqref{Poincare} (see Lemma 2.7 in \cite{FKV}), and then the existence of solutions follows by the same argument above.


\vspace{2mm}

Finally, notice that when $u$ is regular enough then for all $\varphi\in X$ we have
\[\begin{split}
(u,\varphi)&=\frac12\int_{\R^n}\int_{\R^n}\bigl(u(x)-u(z)\bigr)\bigl(\varphi(x)-\varphi(z)\bigr)K(x-z)dx\,dz\\
&=\frac12\,\textrm{PV}\int_{\R^n}\int_{\R^n}\bigl(u(x)-u(z)\bigr)\varphi(x)K(x-z)dx\,dz+\\
&\qquad\qquad\qquad\qquad\qquad\qquad  +\frac12\,\textrm{PV}\int_{\R^n}\int_{\R^n}\bigl(u(z)-u(x)\bigr)\varphi(z)K(x-z)dx\,dz\\
&=\frac12\int_{\R^n}Lu(x)\,\varphi(x)dx+\frac12\int_{\R^n}Lu(z)\varphi(z)dz\\
&=\int_\Omega Lu\,\varphi,
\end{split}\]
where we used that $K(y)=K(-y)$ and that $\varphi\equiv0$ in $\R^n\setminus\Omega$.

This means that, when $u$ is regular enough, the weak formulation \eqref{weak} reads as
\[\int_\Omega Lu\,\varphi=\int_\Omega f\,\varphi\qquad \textrm{for all}\ \varphi\in X,\]
and thus it is equivalent to $Lu=f$ in $\Omega$.

\begin{rem}
We showed in this Section how to prove the existence of \emph{weak} solutions to \eqref{pb} by variational methods.
On the other hand, an alternative approach is to use Perron's method to establish the existence of \emph{viscosity} solutions; see \cite{BCI2,CS}.
As we will see, when the right hand side $f$ is H\"older continuous then all solutions are classical solutions (in the sense that the operator $L$ can be evaluated pointwise), and thus the notions of weak and viscosity solutions coincide; see \cite{SV} for more details.
\end{rem}

\section{Comparison principle}
\label{max-princ}

In this section we show the maximum principle and the comparison principle for weak solutions to \eqref{pb}.

As in the classical case of the Laplacian $-\Delta$, the maximum principle essentially relies on the fact that $Lu(x_0)\geq 0$ whenever $u$ has a maximum at $x_0$.
In case of local equations, this is true for any local maximum, while in nonlocal equations this is only true when $u$ has a global maximum at $x_0$.
In this case, the inequality $Lu(x_0)\geq0$ follows simply from the expression
\[Lu(x_0)=\frac12\int_{\R^n}\bigl(2u(x_0)-u(x_0+y)-u(x_0-y)\bigr)K(y)dy\]
and the fact that $u(x_0)\geq u(x_0\pm y)$ for any $y\in \R^n$ (which holds whenever $u$ has a global maximum at $x_0$).

Notice also that when $K(y)>0$ in $\R^n$, this argument already yields a strong maximum principle, since one has $Lu(x_0)>0$ unless $u\equiv ctt$.

The previous considerations work when $u$ is regular enough, so that $Lu$ can be evaluated pointwise.
In case of weak solutions $u$ to \eqref{pb}, the proof of the maximum principle goes as follows.
For simplicity, we do it in the case $K(y)>0$, but the same could be done for more general kernels $K(y)\geq0$.

\begin{prop}\label{prop-max-princ}
Let $L$ be any operator of the form \eqref{L}, with $K(y)>0$ in $\R^n$.
Let $u$ be any weak solution to \eqref{pb}, with $f\geq0$ in $\Omega$ and $g\geq0$ in $\R^n\setminus\Omega$.
Then, $u\geq0$ in $\Omega$.
\end{prop}

\begin{proof}
Recall that $u$ is a weak solution of \eqref{pb} if
\begin{equation}\label{max-1}
\int\int_{\R^{2n}\setminus(\Omega^c\times\Omega^c)}\bigl(u(x)-u(z)\bigr)\bigl(\varphi(x)-\varphi(z)\bigr)K(z-x)dx\,dz=\int_\Omega f\varphi
\end{equation}
for all $\varphi\in H_K(\R^n)$ with $\varphi\equiv0$ in $\R^n\setminus\Omega$.

Write $u=u^+-u^-$ in $\Omega$, where $u^+=\max\{u,0\}\chi_\Omega$ and $u^-=\max\{-u,0\}\chi_\Omega$.
We will take $\varphi=u^-$, assume that $u^-$ is not identically zero, and argue by contradiction.

Indeed, since $f\geq0$ and $\varphi\geq0$, then we clearly have that
\begin{equation}\label{max-2}
\int_\Omega f\varphi\geq0.
\end{equation}
On the other hand, we have that
\[\begin{split}
\int\int_{\R^{2n}\setminus(\Omega^c\times\Omega^c)}&\bigl(u(x)-u(z)\bigr)\bigl(\varphi(x)-\varphi(z)\bigr)K(z-x)dx\,dz=\\
&\qquad = \int_\Omega\int_\Omega \bigl(u(x)-u(z)\bigr)\bigl(u^-(x)-u^-(z)\bigr)K(z-x)dx\,dz\,+\\
&\qquad\qquad+2\int_\Omega dx\int_{\Omega^c}\bigl(u(x)-g(z)\bigr)u^-(x)K(z-x)dz.
\end{split}\]
Moreover, $\bigl(u^+(x)-u^+(z)\bigr)\bigl(u^-(x)-u^-(z)\bigr)\leq0$, and thus
\[\begin{split}
&\int_\Omega\int_\Omega \bigl(u(x)-u(z)\bigr)\bigl(u^-(x)-u^-(z)\bigr)K(z-x)dx\,dz\\
&\qquad\qquad \leq -\int_\Omega\int_\Omega \bigl(u^-(x)-u^-(z)\bigr)^2K(z-x)dx\,dz<0.\end{split}\]
Also, since $g\geq0$ then
\[\int_\Omega dx\int_{\Omega^c}\bigl(u(x)-g(z)\bigr)u^-(x)K(z-x)dz\leq 0.\]

Therefore, we have shown that
\[\int\int_{\R^{2n}\setminus(\Omega^c\times\Omega^c)}\bigl(u(x)-u(z)\bigr)\bigl(\varphi(x)-\varphi(z)\bigr)K(z-x)dx\,dz<0,\]
and this contradicts \eqref{max-1}-\eqref{max-2}.
\end{proof}

Of course, once we have the maximum principle, the comparison principle follows immediately.

\begin{cor}\label{comparison}
Let $L$ be any operator of the form \eqref{L}, with $K(y)>0$ in $\R^n$.
Let $u_1$ and $u_2$ be weak solutions to
\[\left\{ \begin{array}{rcll}
L u_1 &=&f_1&\textrm{in }\Omega \\
u_1&=&g_1&\textrm{in }\R^n\backslash\Omega\end{array}\right.
\qquad \textrm{and}\qquad
\left\{ \begin{array}{rcll}
L u_2 &=&f_2&\textrm{in }\Omega \\
u_2&=&g_2&\textrm{in }\R^n\backslash\Omega.\end{array}\right.\]
Assume that $f_1\geq f_2$ and $g_1\geq g_2$.
Then, $u_1\geq u_2$.
\end{cor}

\begin{proof}
Just apply Proposition \ref{prop-max-princ} to $u=u_1-u_2$.
\end{proof}

This allows us to use barriers, which in turn can be used to show that solutions $u$ to \eqref{pb} belong to $L^\infty(\Omega)$ whenever $f$ and $g$ are bounded.
This is what we do in the next Section.

\section{Barriers and $L^\infty$ bounds}
\label{barriers}

We provide in this Section an $L^\infty$ estimate of the form
\begin{equation}\label{L^infty-bound}
\|u\|_{L^\infty(\Omega)}\leq \|g\|_{L^\infty(\R^n\setminus\Omega)}+C\|f\|_{L^\infty(\Omega)}
\end{equation}
for solutions to \eqref{pb}.

To do it, we assume first $K(y)>0$ in $\R^n$, as in the previous Section.
In this case, the construction of a barrier is quite simple.

\begin{lem}\label{barrier1}
Let $L$ be an operator of the form \eqref{L}, with $K(y)>0$ in $\R^n$.
Then, there exists a function $w\in C^\infty_c(\R^n)$ such that
\[\left\{ \begin{array}{rcll}
L w &\geq&1&\textrm{in }\Omega  \\
w&\geq&0&\textrm{in }\R^n\backslash\Omega\\
w &\leq&C&\textrm{in }\Omega.\end{array}\right.\]
The constant $C$ depends only on the kernel $K$ and $\textrm{diam}(\Omega)$.
\end{lem}

Once we have this barrier, the $L^\infty$ bound \eqref{L^infty-bound} follows from the comparison principle, as shown next.

\begin{cor}\label{L^infty-bound-cor}
Let $L$ be any operator of the form \eqref{L}, with $K(y)>0$ in $\R^n$.
Let $u$ be any weak solution of \eqref{pb}.
Then,
\[\|u\|_{L^\infty(\Omega)}\leq \|g\|_{L^\infty(\R^n\setminus\Omega)}+C\|f\|_{L^\infty(\Omega)},\]
where $C$ is the constant in Lemma \ref{barrier1}.
\end{cor}

\begin{proof}
Let $v(x)=\|g\|_{L^\infty}+\|f\|_{L^\infty}w(x)$, where $w$ is given by Lemma \ref{barrier1}.
Then, we clearly have $Lu\leq Lv$ in $\Omega$, and $g\leq v$ in $\R^n\setminus\Omega$.

Thus, by the comparison principle, we have $u\leq v$ in $\Omega$.
In particular, $u\leq \|g\|_{L^\infty}+C\|f\|_{L^\infty}$ in $\Omega$.

Applying the same argument to $(-u)$, we find that $-u\leq \|g\|_{L^\infty}+C\|f\|_{L^\infty}$, and hence the result follows.
\end{proof}

We now construct the barrier.

\begin{proof}[Proof of Lemma \ref{barrier1}]
Let $B_R$ be any large enough ball such that $\Omega\subset\subset B_R$, and let $\eta\in C^\infty_c(B_R)$ be such that
\begin{equation}\label{eta}
0\leq \eta\leq 1\quad \textrm{in}\ \R^n, \qquad \eta\equiv1\quad \textrm{in}\ \Omega.
\end{equation}
Then, for each $x\in\Omega$ we have $\eta(x)=\max_{\R^n}\eta$, and thus
\[2\eta(x)-\eta(x+y)-\eta(x-y)\geq \eta(x)-\eta(x+y)\geq 0.\]
Hence, writing $z=x+y$, we have
\[L\eta(x)\geq\int_{\R^n}\bigl\{\eta(x)-\eta(z)\bigr\}K(x-z)dz\geq \int_{\R^n\setminus B_R} K(x-z)dz,\]
where we have used that $\eta(x)-\eta(z)=1$ in $\R^n\setminus B_R$.
Now, we notice that
\[\int_{\R^n\setminus B_R} K(x-z)dz= \int_{\R^n\setminus (x+B_R)}K(y)dy\geq \int_{\R^n\setminus B_{2R}}K(y)dy=c>0\]
for some positive constant $c$.

Hence, we have $L\eta\geq c>0$ in $\Omega$.
Taking $w=\frac{1}{c}\eta$, we will have that
\[\left\{ \begin{array}{rcll}
L w &\geq&1&\textrm{in }\Omega  \\
w&\geq&0&\textrm{in }\R^n\backslash\Omega\\
w &\leq&\frac{1}{c}&\textrm{in }\Omega,\end{array}\right.\]
as desired.
\end{proof}

Note that the previous proof works not only for all kernels satisfying $K(y)>0$ in $\R^n$, but also for all kernels $K$ satisfying
\begin{equation}\label{cond-barrier}
\int_{\R^n\setminus B_{R}}K(y)dy>0
\end{equation}
for every ball $B_R$.
In particular, the above construction works for all stable operators~\eqref{stable}.

Hence, we see that the desired barrier is quite easy to construct, and does not even depend much on the class of kernels we are considering.
Essentially, thanks to the nonlocal character of the operator, any function $\eta$ as in \eqref{eta} works.

\begin{rem}
For operators $L$ that do not satisfy \eqref{cond-barrier}, one can still manage to prove an analogous result.
Indeed, if \eqref{cond-barrier} does not hold, then $K=0$ a.e. outside a ball~$B_R$.
This means that $K$ have compact support.
In this case, one can take any function $\eta\geq0$ which is strictly concave in a large ball $B_M$ and zero outside $B_M$.
Then, if $M$ is large enough (so that $\Omega+\textrm{supp}(K)\subset B_M$ and hence $2\eta(x)-\eta(x+y)-\eta(x-y)>0$ therein), one will have $L\eta\geq c>0$ in $\Omega$.
\end{rem}

Finally, to end this Section, we give an important explicit solution for the class of stable operators \eqref{stable}.

\begin{lem}\label{explicit-sol}
Let $L$ be any stable operator of the form \eqref{L}-\eqref{stable}.
Then, the function\footnote{Here $z_+$ denotes the positive part of the number $z$, i.e., $z_+=\max\{z,0\}$.}
\[u_0(x):=\left(1-|x|^2\right)^s_+\]
solves
\[\left\{ \begin{array}{rcll}
L u_0 &=&c&\textrm{in }B_1  \\
u_0&=&0&\textrm{in }\R^n\backslash B_1\end{array}\right.\]
for some positive constant $c>0$.
\end{lem}

This explicit solution will be used in Section \ref{boundary} to construct a subsolution and establish a Hopf Lemma for this class of operators.

Lemma \ref{explicit-sol} was first established for $(-\Delta)^s$ in dimension $n=1$ by Getoor \cite{G}, and requires quite fine computations; see also the work of Dyda \cite{Dyda}.

Once this is established in dimension $n=1$, the result in dimension $n$ and for general stable operators $L$ follows by just writing
\[Lu_0(x)=\frac14\int_{S^{n-1}}a(\theta)\left(\int_{\R}\frac{2u_0(x)-u_0(x+\tau\theta)-u_0(x-\tau\theta)}{|\tau|^{1+2s}}\,d\tau\right)d\theta,\]
and using the result in dimension 1 for each direction $\theta$.
It is important to notice that the 1D functions $\tau\mapsto u_0(x+\tau\theta)$ are exactly a rescaled version of $(1-|\tau|^2)^s_+$, which solves the equation in dimension $n=1$.
Using this, one gets that $Lu_0=c$ in $B_1$, with $c=c_s\int_{S^{n-1}}a(\theta)d\theta$, and $c_s$ depending only on $s$.

\section{Interior regularity}
\label{interior}

For second order equations, the classical Schauder estimate for $-\Delta u=f$ in $B_1$ establishes that
\begin{equation}\label{estimate-2}
\|u\|_{C^{2+\alpha}(B_{1/2})}\leq C\left(\|f\|_{C^\alpha(B_1)}+\|u\|_{L^\infty(B_1)}\right)
\end{equation}
whenever $\alpha>0$ is not an integer.
Thus, it immediately follows from this estimate that solutions to the Dirichlet problem
\[\left\{ \begin{array}{rcll}
-\Delta u &=&f&\textrm{in }\Omega \\
u&=&g&\textrm{on }\partial\Omega
\end{array}\right.\]
are $C^{2+\alpha}$ inside $\Omega$ whenever $f\in C^\alpha$ and $g$ is bounded.

\vspace{2mm}

For nonlocal equations of order $2s$, one would expect a similar estimate, in which the norm $\|u\|_{C^{2s+\alpha}(B_{1/2})}$ is controlled by $\|f\|_{C^\alpha(B_1)}$ for solutions to $Lu=f$ in $B_1$.
It turns out that, due to the nonlocality of the equation, the norm $\|u\|_{L^\infty(B_1)}$ in \eqref{estimate-2} has to be replaced by a global norm of $u$, i.e., a norm that controls $u$ in the whole~$\R^n$.

\subsection{Regular kernels}

For the fractional Laplacian $L=(-\Delta)^s$ ---which corresponds to $K(y)=c|y|^{-n-2s}$ in \eqref{L}---, this estimate reads as
\begin{equation}\label{estimate-2s}
\|u\|_{C^{2s+\alpha}(B_{1/2})}\leq C\left(\|f\|_{C^\alpha(B_1)}+\|u\|_{L^\infty(\R^n)}\right),
\end{equation}
and holds whenever $\alpha+2s$ is not an integer; see for example \cite{L,RS-Dir}.

The same estimate \eqref{estimate-2s} holds for more general operators of the form \eqref{L}-\eqref{rough} or \eqref{L}-\eqref{stable} under the extra assumption that the kernels $K(y)$ are $C^\alpha$ outside the origin; see \cite[Corollary 1.2]{Se2} and \cite[Corollary 3.5]{RS-stable} for more details.

\vspace{2mm}

This means that, in this case of regular kernels, solutions to \eqref{pb} are $C^{2s+\alpha}$ inside~$\Omega$ whenever $f\in C^\alpha$ and $g$ is bounded.

\vspace{2mm}

However, for general operators \eqref{L}-\eqref{rough} or \eqref{L}-\eqref{stable} (with no further regularity assumption on the kernel $K$), the estimate \eqref{estimate-2s} is not true anymore, and one needs a stronger norm of $u$ in the right hand side, as explained next.

\subsection{Non-regular kernels}

When the kernels $K(y)$ are not regular outside the origin, one has the following estimate.

\begin{thm}[\cite{Se2,RS-stable}]\label{thm-interior}
Let $L$ be any operator of the form \eqref{L} with kernel $K$ of the form either \eqref{rough} or \eqref{stable}.
Let $\alpha>0$ be such that $\alpha+2s$ is not an integer, and let $u\in L^\infty(\R^n)$ be any weak solution to $Lu=f$ in $B_1$.
Then,
\begin{equation}\label{estimate-2s-rough}
\|u\|_{C^{2s+\alpha}(B_{1/2})}\leq C\left(\|f\|_{C^\alpha(B_1)}+\|u\|_{C^\alpha(\R^n)}\right),
\end{equation}
for some constant $C$ that depends only on $n$, $s$, and $\lambda$, and $\Lambda$.
\end{thm}

It is important to remark that the previous estimate is valid also in case $\alpha=0$ (in which the $C^\alpha$ norm has to be replaced by the $L^\infty$); see Theorem 1.1 in \cite{RS-stable} for more details.

The proof of this estimate uses a refined version of the blow-up and compactness argument first introduced by Serra in \cite{Se}.
The estimate \eqref{estimate-2s-rough} was established in \cite{Se2} for the class of kernels \eqref{rough}\footnote{The estimate in \cite{Se2} is a much stronger result, which holds for fully nonlinear equations and also for $x$-dependent equations.}, and in \cite{RS-stable} for the class \eqref{stable}.

With no further regularity assumption on the kernels $K$, the estimate \eqref{estimate-2s-rough} is sharp, in the sense that the norm $\|u\|_{C^\alpha(\R^n)}$ can not be replaced by a weaker one.
More precisely, one has the following.

\begin{prop}[\cite{Se2,RS-stable}]
Let $s\in (0,1)$, $\alpha\in (0,s]$.
Then, for any small $\epsilon>0$ there exists an operator of the form \eqref{L}-\eqref{stable}, and a solution $u$ to $Lu=0$ in $B_1$ such that $u\in C^{\alpha-\epsilon}(\R^n)$, but $u\notin C^{2s+\alpha}(B_{1/2})$.
Moreover, the same happens for the operators \eqref{L}-\eqref{rough}.
\end{prop}

Thus, even when $f\equiv0$, solutions $u$ to
\[\left\{ \begin{array}{rcll}
Lu &=&0&\textrm{in }B_1 \\
u&=&g&\textrm{in }\R^n\setminus B_1
\end{array}\right.\]
are in general no better than $C^{2s}(B_{1/2})$ if $g$ is not better than $L^\infty(\R^n)$.
This means that the interior regularity of solutions to \eqref{pb} depends on the regularity of $g$ in $\R^n\setminus\Omega$ when the kernels are not regular.

\vspace{2mm}

Finally notice that, even in case that $g$ is $C^\infty$ in all of $\R^n\setminus\Omega$, this is not enough to deduce that $u$ is regular inside $\Omega$!
Indeed, because of \eqref{estimate-2s-rough}, one has to control the term $\|u\|_{C^\alpha(\R^n)}$ in order to have a $C^{2s+\alpha}$ estimate inside $\Omega$.
For this, it is not enough to have $\|g\|_{C^\alpha(\R^n\setminus\Omega)}\leq C$, but one also needs to control the regularity of $u$ across~$\partial\Omega$.
In other words, the boundary regularity of $u$ is needed.
We will come back to this in Section \ref{further-interior}.

\section{Boundary regularity}
\label{boundary}

In this section we study the boundary regularity of solutions to
\begin{equation}\label{pb0}
\left\{ \begin{array}{rcll}
L u &=&f&\textrm{in }\Omega \\
u&=&0&\textrm{in }\R^n\backslash\Omega.
\end{array}\right.\end{equation}
We will first look at the optimal H\"older regularity of solutions $u$ near $\partial\Omega$, to then see more fine results for the class of kernels \eqref{stable}.

\begin{rem}\label{rem-bdry}
Once \eqref{pb0} is well understood, the boundary regularity of solutions to \eqref{pb} follows from the results for \eqref{pb0}, at least when the exterior data $g$ in \eqref{pb} is regular enough.

Indeed, assume that $u$ is a solution to \eqref{pb} and that $g\in C^{2s+\gamma}(\R^n\setminus\Omega)$.
We may extend $g$ to a function in $\R^n$ satisfying $g\in C^{2s+\gamma}(\R^n)$, and consider $\tilde u=u-g$.
Then $\tilde u$ satisfies $\tilde u\equiv0$ in $\R^n\setminus\Omega$, and $L\tilde u=f-Lg=:\tilde f$ in $\Omega$.
Hence, $\tilde u$ solves \eqref{pb0} with $u$ and $f$ replaced by $\tilde u$ and $\tilde f$, respectively.
\end{rem}

\vspace{1mm}

\subsection{H\"older regularity for $u$}

As we saw in Section \ref{barriers}, given any stable operator \eqref{L}-\eqref{stable}, the function
\[u_0(x)=\bigl(1-|x|^2\bigr)^s_+\]
is an explicit solution to \eqref{pb0} in $\Omega=B_1$; see Lemma \ref{explicit-sol}.

The function $u_0$ belongs to $C^s(\overline{B_1})$, but
\[u_0\notin C^{s+\epsilon}(\overline{B_1})\quad \textrm{for any}\quad \epsilon>0.\]
This means that, even in the simplest case $K(y)=|y|^{-n-2s}$, one can not expect solutions to be better than $C^s(\overline\Omega)$.

\vspace{2mm}

For the class of kernels \eqref{stable}, using the explicit solution $u_0$ and similar barriers, it is possible to show that, when $\Omega$ is $C^{1,1}$, solutions $u$ satisfy
\[\qquad|u|\leq Cd^s\quad\textrm{in}\ \Omega,\qquad\qquad d(x)=\textrm{dist}(x,\R^n\setminus\Omega).\]
Combining this bound with the interior estimates, one gets the following.

\begin{prop}[Optimal H\"older regularity, \cite{RS-stable}]
Let $L$ be any operator of the form \eqref{L}-\eqref{stable}, and $\Omega$ any bounded $C^{1,1}$ domain.
Let $f\in L^\infty(\Omega)$, and $u$ be the weak solution of \eqref{pb0}.
Then,
\[\|u\|_{C^s(\overline\Omega)}\leq C\|f\|_{L^\infty(\Omega)}\]
for some constant $C$ that depends only on $n$, $s$, $\Omega$, $\Lambda$, and $\lambda$.
\end{prop}

Furthermore, using similar barriers one can show:

\begin{lem}[Hopf's Lemma]\label{Hopf}
Let $L$ be any operator of the form \eqref{L}-\eqref{stable}, and $\Omega$ any bounded $C^{1,1}$ domain.
Let $u$ be any weak solution to \eqref{pb0}, with $f\geq0$.
Then, either
\[u\geq c\,d^s\quad\textrm{in}\ \Omega\quad\textrm{for some}\ c>0\]
or $u\equiv0$ in $\Omega$.
\end{lem}

\vspace{2mm}

Thus, the boundary regularity of solutions depends essentially on the construction of suitable barriers, and the behavior of these barriers near $\partial\Omega$ depends on the class of kernels under consideration.

When the kernels $K(y)$ are of the form \eqref{stable}, we have seen that solutions behave like $d^s$ near the boundary $\partial\Omega$.
However, for the class of kernels \eqref{rough}, the only barriers that one can construct behave like $d^{\alpha_0}$ near the boundary, for some $0<\alpha_0<s$.
This means that for this class of operators one can only prove $u\in C^{\alpha_0}(\overline\Omega)$.


\subsection{Regularity of $u/d^s$}

We have seen that, when
\begin{equation}\label{kernel-stable}
K(y)=\frac{a(y/|y|)}{|y|^{n+2s}},
\end{equation}
then all solutions $u$ behave like $d^s$ near the boundary.
For this class of kernels, much more can be said about the regularity of solutions near $\partial\Omega$.

Indeed, the quotient $u/d^s$ is not only bounded, but it is also H\"older continuous \emph{up to the boundary}.

This regularity of $u/d^s$ yields in particular the existence of the limit
\[\frac{u}{d^s}(z):=\lim_{\Omega \ni x\rightarrow z}\frac{u(x)}{d^s(x)}\]
for all $z\in \partial\Omega$.
This function $u/d^s$ on $\partial\Omega$ plays sometimes the role that the normal derivative $\partial u/\partial\nu$ plays in second order equations.
For example, it appears in overdetermined problems \cite{FJ}, integration by parts formulas \cite{RSV}, and free boundary problems \cite{CRS}.

The first proof of this result was given in \cite{RS-Dir} for the case $K(y)=|y|^{-n-2s}$, and more recently the result has been improved by Grubb in \cite{Grubb,Grubb2} and by the author and Serra in \cite{RS-K,RS-stable}.
These results may be summarized as follows.

\begin{thm}[\cite{RS-stable,RS-K,Grubb,Grubb2}]
Let $s\in(0,1)$, $L$ be any operator of the form \eqref{L}-\eqref{stable}, $\Omega$~be any bounded domain, and $u$ be any solution to \eqref{pb0}, with $f\in L^\infty(\Omega)$.

Then, depending on the regularity of the function $a$ in \eqref{kernel-stable}, we have the following.
\begin{itemize}
\item[(i)] If $a$ is any measure and $\Omega$ is $C^{1,1}$, then
\[\|u/d^s\|_{C^{s-\epsilon}(\overline\Omega)}\leq C\|f\|_{L^\infty(\overline\Omega)}.\]
\item[(ii)] If $a\in C^{1,\gamma}(S^{n-1})$ and $\Omega$ is $C^{2,\gamma}$ for some {small} $\gamma>0$, then
\[\|u/d^s\|_{C^{s+\gamma}(\overline\Omega)}\leq C\|f\|_{C^\gamma(\overline\Omega)}\]
whenever $s+\gamma$ is not an integer and $f\in C^\gamma(\overline\Omega)$.
\item[(iii)] If $a\in C^\infty(S^{n-1})$, $\Omega$ is $C^\infty$, and $f\in C^\gamma(\overline\Omega)$, then
\[\|u/d^s\|_{C^{s+\gamma}(\overline\Omega)}\leq C\|f\|_{C^\gamma(\overline\Omega)}\]
for {all} $\gamma\in (0,\infty)$ such that $s+\gamma$ is not an integer.
\end{itemize}
\end{thm}

Part (i) corresponds to Theorem 1.2 in \cite{RS-stable}, part (ii) was established in \cite{RS-K} in the more general context of fully nonlinear equations, and part (iii) was established in \cite{Grubb,Grubb2} for all pseudodifferential operators satisfying the $\mu$-transmission property. (In the excepted cases of (iii), more information is given in \cite{Grubb2} in terms of H\"older-Zygmund spaces $C^k_*$, when $s + \gamma$ is  an integer.)

\vspace{2mm}

As in the interior regularity estimates, the regularity of the kernel $K$ affects the regularity of the solution $u$.
In this direction, even if one assumes that $\Omega$ is $C^\infty$ and $f\in C^\infty(\overline\Omega)$, the result in part (i) can not be improved if the kernels are singular; see \cite{RS-stable}.

On the other hand, notice that these results yield $(u-g)/d^s\in C^{s+\gamma}(\overline\Omega)$ for solutions $u$ to \eqref{pb}; see Remark \ref{rem-bdry}.
This gives a description of $u$ near $\partial\Omega$ up to order $2s+\gamma$.
For example, if $s+\gamma<1$, this means that there exists a function $b\in C^\gamma(\partial\Omega)$ such that
\[\bigl|u(x)-g(x)-b(z)d^s(x)\bigr|\leq C|x-z|^{2s+\gamma}\qquad \textrm{for}\ z\in \partial\Omega,\quad x\in B_1(z).\]
When $s+\gamma\in (1,2)$ one has a similar expansion with an additional term of order $d^{s+1}$, and more terms appear for higher values of $\gamma$.

\section{Further interior regularity}
\label{further-interior}

As we saw in the previous Section, solutions $u$ to \eqref{pb0} are $C^s$ or $C^{\alpha_0}$ up to the boundary when $K(y)$ is of the form \eqref{rough} or \eqref{stable}.
Moreover, the same happens for solutions to \eqref{pb} whenever the exterior data $g$ is regular enough.

Hence, this means that solutions $u$ will be globally H\"older continuous, and therefore we may apply the estimate
\begin{equation}\label{***}
\|u\|_{C^{\alpha+2s}(B_{1/2})}\leq C\bigl(\|f\|_{C^\alpha(\overline\Omega)}+\|u\|_{C^\alpha(\R^n)}\bigr)
\end{equation}
to get that solutions $u$ are $C^{2s+\alpha}$ inside $\Omega$ for some $\alpha>0$ ---more precisely, $\alpha=s$ in case \eqref{stable}, and $\alpha=\alpha_0$ in case \eqref{rough}.
This is enough to ensure that, when $f\in C^\alpha$, any weak solution is a classical solution, in the sense that the operator $L$ can be evaluated pointwise.

\vspace{2mm}

The natural question then is: are solutions more regular than this?

\vspace{2mm}

For the class of kernels \eqref{stable} the interior estimate \eqref{***} does not give more than $u\in C^{3s}_{\textrm{loc}}(\Omega)$, and we know that \eqref{***} is sharp.
Still, if one considers the solution $u$ to
\begin{equation}\label{pb01}
\left\{ \begin{array}{rcll}
L u &=&1&\textrm{in }\Omega \\
u&=&0&\textrm{in }\R^n\backslash\Omega
\end{array}\right.\end{equation}
in a $C^\infty$ domain $\Omega$, one a priori does not know if solutions are more regular than~$C^{3s}$.
This question has been answered by the author and Valdinoci in \cite{RV}:

\begin{thm}[\cite{RV}]
Let $L$ be any operator of the form \eqref{L}-\eqref{stable}, and $\Omega$ be any bounded domain.
Then,
\begin{itemize}
\item[(i)] If $a$ is any measure and $\Omega$ is convex, then $u\in C^{1+3s-\epsilon}_{\textrm{loc}}(\Omega)$ for all $\epsilon>0$.

\item[(ii)] If $a\in L^\infty(S^{n-1})$ and $\Omega$ is $C^{1,1}$, then $u\in C^{1+3s-\epsilon}_{\textrm{loc}}(\Omega)$ for all $\epsilon>0$.
\item[(iii)] There is a (nonconvex) $C^\infty$ domain $\Omega$, and and an operator $L$ of the form \eqref{L}-\eqref{stable}, for which the solution $u$ to \eqref{pb01} does not belong to $C^{3s+\epsilon}_{\textrm{loc}}(\Omega)$ for any $\epsilon>0$.
\end{itemize}
\end{thm}

Thus, solutions have different regularity properties depending on the {shape} of $\Omega$.
Indeed, for operators $L$ with singular kernels, the convexity of the domain gives one more order of differentiability on the solution.
Also, notice that when $a\in L^\infty$, then we gain interior regularity on the solution by assuming that the domain $\Omega$ is $C^{1,1}$.

For kernels $K$ in the class \eqref{rough}, one would expect solutions $u$ to be $C^{1+2s+\alpha_0-\epsilon}$ in the interior of $C^{1,1}$ domains.

\end{document}